\documentclass[12pt]{amsart}
\usepackage[colorlinks,citecolor=blue]{hyperref}
\usepackage{amsmath}
\usepackage{amsthm}
\usepackage{textcomp}
\usepackage{amsfonts}
\usepackage{txfonts}
\usepackage{mathtools}
\usepackage{amssymb}
\usepackage{newlfont}
\usepackage{galois}
\usepackage{graphicx}
\numberwithin{equation}{section}

\newtheorem{theorem}{Theorem}[section]
\newtheorem{lemma}{Lemma}[section]
\newtheorem{remark}{Remark}[section]
\newtheorem{corollary}{Corollary}[section]
\usepackage{lineno,hyperref}

\begin{document}
\title{Gradient Estimates For A Class of Elliptic and Parabolic Equations on Riemannian Manifolds}
\author{Jie Wang}
\address{School of Mathematical Sciences, UCAS, Beijing 100190, China; Institute of Mathematics,
The Academy of Mathematics and Systems of Sciences, Chinese Academy of Sciences}
\email{wangjie9math@163.com}

\begin{abstract}
Let $(N, g)$ be a complete noncompact Riemannian manifold with Ricci curvature bounded from below. In this paper, we study the gradient estimates of positive solutions to a class of nonlinear elliptic equations
$$\Delta u(x)+a(x)u(x)\log u(x)+b(x)u(x)=0$$ on $N$ where $a(x)$ is $C^{2}$-smooth while $b(x)$ is $C^{1}$
and its parabolic counterparts
$$(\Delta-\frac{\partial}{\partial t})u(x,t)+a(x,t)u(x,t)\log u(x,t) + b(x,t)u(x,t)=0$$
on $N\times[0, \infty)$ where $a(x,t)$ and $b(x,t)$ are $C^{2} $ with respect to $x\in N$ while are $C^{1}$ with respect to the time $t$. In contrast with lots of similar results, here we do not assume the coefficients of equations are constant, so our results can be viewed as extensions to several classical estimates.
\end{abstract}
\keywords{gradient estimates; nonlinear equations; maximum principle}
\maketitle

\section{Introduction}
Recently, one pay attention to studying the following elliptic equation defined on a complete, noncompact Riemannian manifold $(N, g)$
\begin{equation}\label{first}
\Delta u(x) + a(x)u(x)\left(\log u(x)\right)^{\alpha}+b(x)u(x)=0,
\end{equation}
where $\alpha\in\mathbb{R}$ and $a(x), b(x)\in C^2(N)$. The equation is linked with gradient Ricci solitons, for example, see \cite{CL, C-C*, Ma, Yang} for detailed explanations.

On the other hand, it is closely related to log-Sobolev constants of Riemmannian manifolds (see \cite{Chu-Y}). Recall that, log-Sobolev constants $S_N$, associated to a closed Riemannian manifold $(N, g)$, are the smallest positive constants such that the logarithmic-Sobolev inequality
$$\int_Nu^2\log u^2dN\leq S_N\int_N|\nabla u|^2dN$$
for all smooth function $u$ defined on $N$ with $\int_Nu^2dN=\mbox{Vol}(N)$. If $\psi$ is an extreamal function which achieves the log-Sobolev constant and satisfies $\int_N\psi^2dN=\mbox{Vol}(N)$, i.e.
$$S_N=\frac{\int_N|\nabla \psi|^2dN}{\int_N\psi^2\log \psi^2dN}=\inf_{\phi\neq 0\in C^1(N)}\frac{\int_N|\nabla \phi|^2dN}{\int_N\phi^2\log \phi^2dN},$$
by variation we know that $\psi$ satisfies the Euler-Lagrange equation
$$\Delta \psi + S_N\psi\log\psi^2=0.$$

In physics, the logarithmic Schr\"odinger equation written by
$$i\epsilon\frac{\partial\Psi}{\partial t}=-\epsilon^2\Delta\Psi+(W(x)+w)\Psi -\Psi\log|\Psi|^2, \quad\Psi:[0, \infty)\times\mathbb{R}^n\to\mathbb{C},\,\, n\geq1,$$
has also received considerable attention. It is well-known that this class of equation has some important physical applications, such as quantum mechanics, quantum optics, open quantum systems, effective quantum gravity, transport and diffusion phenomena, theory of superfluidity and Bose-Einstein condensation (see \cite{Z} and the references therein). In its turn, standing waves solution, $\Psi$, for this logarithmic Schr\"odinger equation is related to the solutions of the following equation
$$\epsilon^2\Delta u + u\log u^2 - V(x)u =0,$$
where $V(x)$ is a real function on $\mathbb{R}^n$. Many mathematicians have also ever studied the existence and properties of solutions to such elliptic equation on an Euclidean spaces (see \cite{A-J, W-Z} and references therein).

In this paper, we mainly study the gradient estimates of the following nonlinear elliptic equation
 \begin{equation}\label{equ}
 \Delta u(x)+a(x)u(x)\log u(x)+b(x)u(x)=0
 \end{equation}
which is the special case of (\ref{first}), and its parabolic counterpart
\begin{equation}\label{equ*}
(\Delta-\frac{\partial}{\partial t})u(x,t)+a(x,t)u(x,t)\log u(x,t) + b(x,t)u(x,t)=0
\end{equation}
on a complete non-compact Riemannian manifold $N$. In (\ref{equ}), $a(x)$ is $C^{2}$-smooth while $b(x)$ is $C^{1}$ on $N$, and in (\ref{equ*}), both $a(x,t)$ and $b(x,t)$ are $C^{2} $ with respect to $x\in N$ while are $C^{1}$ with respect to the time $t$.

Now, let's recall some previous work related closely to this paper. In the case $a(x)\equiv0$ and $b(x)\equiv0$, \eqref{equ*} is the Laplace equation. The corresponding  gradient estimate of \eqref{equ*} has ever been established by Yau in the very famous paper \cite{yau} (for its generalized version, see the remarkable work \cite{C-Y} due to Cheng-Yau). In 1980s, for the case $a(x)\equiv0$ and $b(x)\neq0$ is a smooth function P. Li and S.T. Yau \cite{LY} proved the well-known Li-Yau estimate for the corresponding heat equation and derived a Harnack inequality. There is a huge literature on the gradient estimates to the solutions of some elliptic and parabolic equations. In general, these estimates have been used to find H\"olders' continuity of solutions, sharp estimate on the fundamental solution, estimate on the principal eigenvalue. Here we would like to refer but a few to \cite{CH, C-C, Chow, KZ, LY, Ma, SY, JS, Wang} and references therein.

For the case $a<0$ is a constant, $b=0$, and $\alpha=1$ in (\ref{first}), that is,
\begin{equation}\label{constant}
\Delta u(x) +au(x) \log u(x) = 0 \quad on \quad N,
\end{equation}
Ma \cite{Ma} studied the gradient estimates of the positive solutions to the above elliptic equation. Later, L. Chen and W. Chen \cite{C-C*} improves the estimate of \cite{Ma} and extends it to the case $a>0$.

The first study of the corresponding heat equation of (\ref{constant}) and related nonlinear heat equations can be traced back to Yang \cite{Yang}, and later by Huang and Ma\cite{HM}, Qian \cite{Q*}, Cao et al \cite{CL}, Zhu and Li \cite{Z-L} and H. Dung and N. Dung \cite{Dung}, who derived various gradient estimates and Harnack estimates and noted the relation to gradient Ricci solitons.

The authors \cite{DungK} have ever researched the bounded positive solutions to a heat equation with $V$-Laplacian and variable coefficients which is closely related to (\ref{equ*}). Recently, Ma and Liu \cite{ML} obtained also the gradient estimate for positive solutions to the nonlinear heat equation (\ref{equ*}) with $a(x,t)\equiv a>0$ a constant and $b(x,t)\equiv b(x)$ is a given smooth positive function on the compact Riemannian manifold of dimension $n$.

Very recently, in \cite{P} Peng studied the equation (\ref{equ}) with constant coefficients and obtains the gradient estimates without assuming $u$ is bounded. Moreover, Peng, Wang and Wei studied \eqref{first} with $a\neq 0$ is a constant and $\alpha\neq 1$ in \cite{PWW1}(also see \cite{PWW}), and obtained some Li-Yau type gradient estimates on the positive solutions to these equations. More precisely, for the case $\alpha=\frac{k_{1}}{2k_{2}+1} \geq 2$, they improved the classical methods and employ some delicate analytic techniques to obtain a gradient bound of a positive solution to \eqref{first} which does not depend on such quantities as the bounds of the solution and the Laplacian of the distance function.

On the other hand, in \cite{Ab} the author also considered the following
\[
\Delta_f u(x) +au(x) (\log u(x))^\alpha = 0
\]
defined on a complete smooth metric measure manifold with weight $e^{-f}$ and Bakry-\'Emery Ricci tensor bounded from below, where $a$ and $\alpha$ are real constants. He obtained the local gradient estimates, which depends on the bound of positive solutions, and prove the global gradient estimates on bounded positive solutions to the equation.

While its parabolic counterpart,
\[
\left( \Delta - q(x,t) - \frac{\partial}{\partial t} \right)u(x,t) = a u(x,t)(\log (u(x,t)))^{\alpha},
\]
where $q(x,t)$ is a $C^{2}$ function, $a$ and $\alpha$ ($\alpha\neq 0$ or $1$) are constants, was also considered by some mathematicians(see \cite{C-C*, Yang, Z-L}). Besides, Wu \cite{W} and Yang and Zhang \cite{Y-Z} also paid attention to a similar nonlinear parabolic equation defined on some kind of smooth metric measure space.

As we have konwn, when $a(x)$ (or $a(x,t)$) is a constant function, there are lots of results about the gradient estimates of the equations (\ref{equ}) and (\ref{equ*}). However, we note that it seems that few mathematicians consider the above equations when $a(x)$ is a function defined on $N$ except for \cite{ML} and \cite{DungK}. This stimulates us to study the gradient estimates for the case $a$ is not a constant. So, in the present paper we are intend to extend the results on gradient estimates of the positive solutions to (\ref{equ}) or (\ref{equ*}) with constant coefficients in \cite{C-C, Ma, Yang} to the case (\ref{equ}) or (\ref{equ*}) with variable coefficients. In particular, we can extend some results in \cite{ML} to the case $N$ is a noncompact manifold, $a(x,t)$ is not a constant function and $b(x,t)$ is a bounded real function.

It seems that the methods in \cite{P} can not be used to deal with the variable coefficient equation here directly. However,  by improving the techniques in \cite{P} (also see \cite{PWW, PWW1}), we can also obtain some similar gradient estimates for the equation (\ref{equ}) with variable coefficient. As for (\ref{equ*}), our basic ideas mainly come from \cite{Yang}, but we need also to adopt some methods and techniques from \cite{PWW, PWW1}.

In order to state our main results, we need to introduce some notations first. Throughout this paper, let the symbol $( g )^{+}$ denote
$$( g )^{+}\equiv\sup\limits_{x\in B_{p}(2R)}\max(g(x),0),$$
or
$$[g ]^{+}\equiv\sup\limits_{(x,t)\in B_{p}(2R)\times(0,\infty)}\max(g(x,t),0).$$
Now we are in the position to state the main results of this paper.

\begin{theorem}\label{thm1}
	Let $(N, g)$ be a complete non-compact Riemannian manifold of dimension n $ \geq $ 2, p is a point on N, and $ B_{p}(2R) $ is a geodesic ball of radius 2R around p and does not intersect with $ {\partial N} $. Let u be positive smooth solutions to equation (\ref{equ}) and $ f=\log u$. Suppose that the Ricci curvature of $N$ denoted by $ Ric_N$ with $Ric_N\geq -K$ on $ B_{p}(2R) $, where $K$ is a nonnegative constant, $|\nabla a|$ and $|\nabla b|$ are bounded. We have
	\medskip
	
	(1). If $a(x)\geq2A_{1}>0$ where $A_{1}$ is a positive constant, $\Delta a$ has a lower bound while b has a upper bound on $ B_{p}(2R)$, then there exist constants $C_1>0$, $C_2>0$ and $M_{1}$ which is a nonnegative constant depending on the bounds of $a(x)$, $b(x)$, $\nabla a(x)$, $\nabla b(x)$, $K$ and the lower bound of $\Delta a(x)$, such that the following local estimate holds true on $B_{p}(R)$
	\begin{equation}\label{1.3}
		\begin{split}
			{\left| \nabla f \right|}^2+(A_1+a)f
			\leq &n\left\lbrace        2B+3n\left( \frac{a^2}{A_1^2}\right) ^+\frac{C_{1}^2}{R^2}+\left( 3K+3\right) \left(\frac{a}{A_1} \right)^+\right\rbrace\\
			&+n\left\lbrace \frac{8}{n}\left( M_1-b\right)^++\left( \frac{\Delta a}{A_1}\right)^++5\left( a\right)^+ \right\rbrace.
		\end{split}
	\end{equation}
	Here $B=\frac{(n-1)(1+\sqrt{K}R)C_{1}^{2}+C_{2}}{R^{2}}$.
	\medskip
	
	(2). If $ A_3\leq a\leq 2A_2<0 $ where $ A_{2} $ and $ A_3 $ are negative constants, $\Delta a(x)$ and $b$ have upper bounds, then there exist constants $C_1>0$, $C_2>$ and $M_{2}$ which is a nonnegative constant depending on the bounds of $a$, $b$, $\nabla a(x)$, $\nabla b(x)$, $K$ and the upper bound of $\Delta a$, such that the local estimate on $B_{p}(R)$ holds true
	\begin{equation}\label{1.4}
		\begin{split}
			{\left| \nabla f \right|}^2+(A_2+a)f
			\leq& n\left\lbrace        2B+3n\left( \frac{a^2}{A_2^2}\right) ^+\frac{C_{1}^2}{R^2}+\left( 3K+3\right) \left(\frac{a}{A_2} \right)^+\right\rbrace\\
			&+n\left\lbrace \frac{8}{n}\left( M_2-b\right)^++\left( \frac{\Delta a}{A_2}\right)^+ + \left(-\frac{3a}{2}\right)^+\right\rbrace.
		\end{split}
	\end{equation}
	Here $B$ is the same as in the above.
\end{theorem}

Using the results of Theorem \ref{thm1}, we obtain the following estimates about the upper or lower bounds of the global solutions to (\ref{equ}).

\begin{corollary}\label{c1.1}
	Let N be a complete non-compact without boundary. If u is a global solution to equation (\ref{equ}) and in addition that $ b\geq b_1$ and $\left| \Delta a\right|  \leq a_1$ on $N$, where $ b_1 $ and $ a_1\geq 0$ are constants.
	
	(1). If the conditions on $a(x)$ and $b(x)$ in (1) of Theorem \ref{thm1} are satisfied on $N$, moreover, $0<2A_1 \leq a \leq A_4 $ where $A_4 $ is a positive constant on $N$, then
	\begin{equation*}
		u\leq e^{n\left\lbrace \frac{(K+1)A_4}{A_1^2}+\frac{8(M_1-b_1)}{3nA_1}+\frac{a_1}{3A_1^2}+\frac{5A_4}{3A_1}\right\rbrace }.
	\end{equation*}
	
	(2). If the conditions on $a(x)$ and $b(x)$ in (2) of Theorem \ref{thm1} are satisfied on $N$, then
	\begin{equation*}
		u\geq e^{n\left\lbrace \frac{(K+1)A_3}{A_2^2}+\frac{8(M_2-b_1)}{3nA_2}-\frac{a_1}{3A_2^2}-\frac{A_3}{2A_2}\right\rbrace }.
	\end{equation*}
\end{corollary}

For heat equation (\ref{equ*}) we obtain the following Li-Yau type estimates:
\begin{theorem}\label{thm2}
Let (N,g) be a complete non-compact Riemannian manifold of dimension  $n\geq $ 2, $p$ is a point on $N$, and $ B_{p}(2R) $ is a geodesic ball of radius $2R$ which does not intersect with ${\partial N} $. Let u be positive smooth solutions to equation (\ref{equ*}) on $N\times\ [0,\infty) $, $0<u\leq D$ and $f=\log\frac{u}{D}$ for some positive constant $D$. Suppose that the Ricci curvature of $N$ denoted by $Ric_N$ with $Ric_N\geq -K$ where $K$ is a nonnegative constant on $B_{p}(2R)$, $ a, b, |\nabla a|, |\nabla b|, |a_t|$ are bounded and $\Delta b$ has a lower bound on $B_{p}(2R)\times(0,\infty)$.
	
Then, there exist $C_1>0$, $C_2>0$ and $A$ which is a positive constant and strictly larger than $[a]^+$ on $B_p(2R)\times(0,\infty)$ and a constant $M$ depending on the bounds of $a, b, |\nabla a|, |\nabla b|, a_t$ and $\Delta b$ on $B_{p}(2R)\times(0,\infty)$, such that the following local estimate holds on $B_p(R)\times(0,\infty)$
\begin{equation}\label{1.5}
\begin{split}
&{\left| \nabla f \right|}^2+(A+a)f-2f_t\\			
&\leq\frac{4n}{t}+4n\left([a]^+ + B+\frac{nC_1^2}{R^2}+\frac{2[M-a\log D]^+}{n}\right) \\
&\hspace*{1.2em}+4n\left\lbrace \frac{1}{2}\left[ -\left( A-2K-2-\left| \log D \right|\right) \right]^++\frac{[\Delta a+a_t]^+}{4(A-[a]^+)} \right\rbrace.
\end{split}
\end{equation}
Here  $B =\frac{(n-1) (1+\sqrt{K}R)C_{1}^{2}+C_{2}}{R^{2}}$.
\end{theorem}

\begin{remark}\label{r1.1}
In fact, from the following proof of theorem 1.2, we can easily see that, if $a(x,t)$ is a constant, we don't need to assume $u$ has a upper bound, especially when $ a(x,t)=0 $, it's even unnecessary to assume $b$ is bounded. Actually, under these special conditions, we can recover the results in \cite{Yang} and the classical Li-Yau estimates in \cite{LY}.
\end{remark}

For applications, we apply the above results to a special logarithmic Schr\"odinger equation
\begin{equation}\label{1.8}
	\Delta u + u\log u^2 + V(x)u =0
\end{equation}
defined on $R^n$ where $V(x)\in C^2(R^n)$ and study its global positive solutions. In this situation, (\ref{1.8}) is equivalent to
\begin{equation}\label{1.9}
	\Delta u + 2u\log u+ V(x)u =0.
\end{equation}
As a consequence, we have the following priori estimates
\begin{theorem}\label{thm3}
Let $u$ be a positive solution to equation (\ref{1.9}), then

(1) if $V(x)$, $\left| \nabla V(x)\right|$ are bounded, then by Theorem \ref{thm1}, $u$ must be bounded:
\begin{equation}\label{1.10}
u\leq e^{\frac{1}{3}\left\lbrace 16n+8\sup\limits_{R^n}\left| V\right|+\frac{8}{3}\sup\limits_{R^n}\left| \nabla V\right|^2-8V\right\rbrace}.
\end{equation}
 Especially, when $V\geq0$ is a constant, from (\ref{1.10}) we know that the upper bound of $u$ is not related to $V$.

 (2) if $V(x)$, $\Delta V(x)$ and $\left| \nabla V(x)\right|$ are bounded, then by Theorem \ref{thm2}, $u$ must be bounded:
 \begin{equation}\label{1.11}
u\leq e^{\left\lbrace 2n+\sup\limits_{R^n}\left| \Delta V\right| +\frac{1}{2}\sup\limits_{R^n}\left| \nabla V\right|+2\sup\limits_{R^n}\left| V\right|\right\rbrace}.
 \end{equation}
 Generally speaking, (\ref{1.10}) is better than (\ref{1.11}) since the former does not need $\Delta V$, but in some special situations, e.g., if $V\leq0$ is a constant, (\ref{1.11})
 is more accurate than (\ref{1.10}).
\end{theorem}

This paper is divided into four sections. Section 2 gives the proof of Theorem \ref{thm1} and Corollary \ref{c1.1}, section 3 gives the proof of Theorem \ref{thm2} and section 4 gives the detailed discussions about Theorem \ref{thm3}.

\section{Proof of Theorem 1.1 and Corollary 1.1}
As usual, our main mathematical tool is the maximum principle, so the first step is to establish the following lemma.
\begin{lemma}
	Let $(N, g)$ satisfy the same conditions as in Theorem \ref{thm1}. Let $u$ be a positive smooth solution to (\ref{equ}),  $w= \log u$, and $G=\left|\nabla w \right|^2+(a+A)w+M$ where $A$ and $M$ are two constants to be determined later.  Then, on $ B_{p}(2R)$ the function $G$ satisfies
	\begin{equation}\label{2.1}
		\begin{split}
			\Delta G\geq&\frac{2G^2}{n}-2\left\langle \nabla G ,\nabla w\right\rangle+G\left\lbrace -\frac{4Aw}{n}+\dfrac{4(b-M)}{n}-2K-2a-2\right\rbrace\\
			&+\left\lbrace \frac{4(M-b)A}{n}+(2K+2)(A+a)+\Delta a\right\rbrace w\\
			&+\frac{2A^2w^2}{n}-A(A+a)w+\frac{2(b-M)^2}{n}+(A+a)(M-b)\\
			&+(2-2A)M+2KM-\left| \nabla a\right|^2- \left| \nabla b\right|^2.
		\end{split}
	\end{equation}
\end{lemma}

\begin{proof}
	First, from the equation (\ref{equ}) and the definition of G, we derive the following two new equations
	\begin{equation}\label{2.2}
		\Delta w=-G+Aw+M-b
	\end{equation}
	and
	\begin{equation}\label{2.3}
		\left|\nabla w \right|^2=G-(A+a)w-M.
	\end{equation}
	Then, using the well-known Bochner formula, we have
	\begin{align}\label{2.4}
		\Delta G&=\Delta \left|\nabla w \right|^2 +\Delta ((A+a)w)\nonumber\\
		&= 2\left\langle \nabla w,\nabla\Delta w\right\rangle +2\left|D^2w \right|^2 +2Ric(\nabla w,\nabla w)+\Delta ((A+a)w).
	\end{align}
	The Cauchy-Schwarz inequality tells us that
	\begin{equation}\label{2.5}
		2\left|D^2w \right|^2\geq\dfrac{2}{n}(\Delta w)^2.
	\end{equation}
	Next, we substitute (\ref{2.2}),  (\ref{2.3}) and  (\ref{2.5}) into (\ref{2.4}) to obtain
	\begin{equation}\label{2.6}
		\begin{split}
			\Delta G\geq &2\left\langle \nabla w,\nabla(-G+Aw+M-b)\right\rangle+\frac{2}{n}(-G+Aw+M-b)^2\\
			&-2K(G-(A+a)w-M)+\Delta ((A+a)w).
		\end{split}
	\end{equation}
	Keeping  (\ref{2.2}) and  (\ref{2.3}) in mind and noting
	\begin{align}\label{2.7}
		\Delta((A+a)w)&=w\Delta a+(A+a)\Delta w+2\left\langle \nabla w,\nabla a\right\rangle\nonumber\\
		&=w\Delta a +(A+a)(-G+Aw+M-b)+2\left\langle \nabla w,\nabla a\right\rangle,
	\end{align}
	we infer from (\ref{2.3}), (\ref{2.6}) and (\ref{2.7}) the following
	\begin{equation}\label{2.8}
		\begin{split}
			\Delta G\geq& \frac{2G^2}{n}-2\left\langle \nabla w,\nabla G\right\rangle+G\left\lbrace -\frac{4Aw}{n}-2a+\dfrac{4(b-M)}{n}-2K\right\rbrace\\&+\dfrac{2A^2w^2}{n}-A(A+a)w +\left\lbrace \frac{4(M-b)A}{n}+2K(A+a)+\Delta a\right\rbrace w\\
			&+\left( -2AM+(A+a)(M-b)+2\left\langle \nabla a,\nabla w\right\rangle\right)\\
			&+\frac{2(b-M)^2}{n}-2\left\langle \nabla b,\nabla w\right\rangle+2KM.
		\end{split}
	\end{equation}
	We can easily see that on $B_p(2R)$ the following inequalities hold true
	\begin{equation}\label{2.9}
		\begin{split}
			2\left\langle \nabla a,\nabla w\right\rangle&\geq-2\left| \nabla a\right| \left| \nabla w\right|\\
			&\geq-2\left( \frac{\left| \nabla a\right|^2}{2}+\frac{\left| \nabla w\right|^2}{2}\right)\\
			&= -\left| \nabla a\right|^2-\left| \nabla w\right|^2\\
			&=(A+a)w-G+M-\left| \nabla a\right|^2,
		\end{split}
	\end{equation}
	and similarly,
	\begin{equation}\label{2.10}
		2\left\langle \nabla b,\nabla w\right\rangle\geq(A+a)w-G+M-\left| \nabla b\right|^2.
	\end{equation}
	By substituting  (\ref{2.9}) and  (\ref{2.10}) into  (\ref{2.8}), we obtain  (\ref{2.1}), hence we accomplish the proof.	
\end{proof}

In order to apply the maximum principle, we need to use the cut-off function introduced by Li-Yau in \cite{LY}. Concretely, let $\psi(r)$ be a nonnegative $C^2$-smooth function on $R^+=\left[0,+\infty \right)$ such that $\psi(r)=1$ for $r\leq 1$ and $\psi(r)=0$ for $ r\geq2 $. Moreover, there exist two positive constants $ C_1 $ and $ C_2 $ such that the derivatives of $\psi(r)$ satisfy the conditions as follows:
\begin{equation}\label{2.11}
	-C_1\psi^{\frac{1}{2}}(r)\leq\psi^{'}(r)\leq 0\quad\quad \mbox{and}\quad\quad -C_2\leq\psi^{''}(r).
\end{equation}

Now, let $\phi(x)=\psi\left( \dfrac{d(x,p)}{R}\right)  $ where $ d(x,p) $ denotes the distance from $p$ to $ x $ on $N$ and it is obvious that $\phi(x)$ is supported in $ B_p(2R)$:
\begin{align*}
	&\phi|_{B_p(R)}=1,\\
	&\phi|_{N\backslash B_p(2R)}=0.
\end{align*}
\hspace*{1em}Furthermore, by Calabi's trick in \cite{Calabi}, we can assume without loss of generality that $\phi$ is smooth on $ B_p(2R)$. Consequently, it follows from (\ref{2.11}) and the Laplacian comparison theorem that
\begin{align}
	&\frac{\left| \nabla  \phi\right|^2 }{\phi}\leq\dfrac{C_1^2}{R^2},\label{2.12}\\
	&\Delta\phi\geq-\dfrac{(n-1)(1+\sqrt{K}R)C_1^2+C_2}{R^2}.\label{2.13}
\end{align}

According to the definition of $\phi(x)$, we know that $\phi G(x)$ is also supported in $B_p(2R)$. Consequently, there exists a point $x_0$ in $B_p(2R)\backslash \partial B_p(2R)$ such that:\\
\begin{equation*}
	\sup\limits_{x\in B_p(2R)}\phi G(x)=\phi G(x_0).
\end{equation*}
Hence, by maximum principle, we have
\begin{equation}\label{2.14}
	\nabla(\phi G)(x_0)=0\quad\quad \mbox{and}\quad\quad \Delta(\phi G)(x_0)\leq0,
\end{equation}
and these imply that at $ x_0 $
\begin{equation}\label{2.15}
	\phi\nabla G=-G\nabla\phi\quad \mbox{and}\quad \phi\Delta G+G\Delta\phi-2G\frac{\left| \nabla\phi\right|^2 }{\phi}\leq0.
\end{equation}

We can also assume without loss of generality that $ \phi G(x_0)  > 0$, otherwise Theorem 1.1 is trivial. After a direct computation,  (\ref{2.12}),  (\ref{2.13}) and  (\ref{2.15}) yield that at $ x_0$
\begin{equation}\label{2.16}
	BG\geq\phi\Delta G
\end{equation}
where $$B=\frac{(n-1)(1+\sqrt{K}R)C_{1}^{2}+C_{2}}{R^{2}}.$$

On the other hand, from  (\ref{2.3}) and  (\ref{2.15}), we also have
\begin{equation}\label{2.17}
	-\left\langle \nabla w,\nabla G\right\rangle=G\left\langle\nabla w,\nabla \phi \right\rangle\geq-G\left| \nabla\phi\right|(G-(A+a)w-M)^{\frac{1}{2}}.
\end{equation}
Eventually, by substituting  (\ref{2.16}) and  (\ref{2.17}) into  (\ref{2.1}), we obtain that there holds true at $ x_0 $
\begin{equation}\label{2.18}
	\begin{split}
		BG&\geq\phi\Delta G\\
		&\geq\frac{2\phi G^2}{n}-2G(G-(A+a)w)^{\frac{1}{2}}\left|\nabla\phi \right| +\phi\left\lbrace \frac{2A^2w^2}{n}-A(A+a)w\right\rbrace\\
		&\hspace*{1.2em}+\phi G\left\lbrace -\frac{4Aw}{n}+\dfrac{4(b-M)}{n}-2K-2a-2\right\rbrace\\
		&\hspace*{1.2em}+\phi w\left\lbrace \frac{4(M-b)A}{n}+(2K+2)(A+a)+\Delta a\right\rbrace \\
		&\hspace*{1.2em}+\phi\left\lbrace (A+a)(M-b)+(2-2A)M+\dfrac{2(b-M)^2}{n}\right\rbrace\\
		&\hspace*{1.2em}+\phi\left\lbrace 2KM-\left| \nabla a\right|^2- \left| \nabla b\right|^2\right\rbrace. 	
	\end{split}
\end{equation}

Now, we are ready to give the complete proof of Theorem \ref{thm1}. In consideration of the whole proof is pretty long, it will be divided into several parts. Since that the following calculations are all considered at the point $ x_0 $, for simplicity, we omit the  $ x_0 $.
\begin{proof}[\textbf{Proof of Theorem 1.1} ]
	First, we give the proof of (1) in Theorem \ref{thm1}.
	
	Since $ a\geq2A_{1}>0 $, $ \left|\nabla a \right|  $, $ \left|\nabla b \right|  $ are bounded, $ \Delta a $ is bounded from below and $ b $ has a upper bound,  there exists a nonnegative constant $ M_1 $ such that on $ B_p(2R)$
	\begin{equation*}
		\frac{2(b-M_1)^2}{n}+(A_1+a)(M_1-b)+(2-2A_1)M_1+2KM_1-\left| \nabla a\right|^2- \left| \nabla b\right|^2\geq0,
	\end{equation*}
	and
	\begin{equation*}
		\frac{4(M_1-b)A_1}{n}+(2K+2)(A_1+a)+\Delta a\geq0.
	\end{equation*}
	Therefore, letting $A=A_1>0$ and $M=M_1\geq0$ in (\ref{2.18}), we obtain
	\begin{equation}\label{2.19}
		\begin{split}
			BG\geq&
			\frac{2\phi G^2}{n}-2G(G-(A_1+a)w)^{\frac{1}{2}}\left|\nabla\phi \right|+\phi\left\lbrace  \frac{2A_1^2w^2}{n}-A_1(A_1+a)w\right\rbrace   \\
			&+\phi G\left\lbrace -\frac{4A_1w}{n}+\dfrac{4(b-M_1)}{n}-2K-2a-2\right\rbrace\\
			&+\phi w\left\lbrace \frac{4(M_1-b)A_1}{n}+(2K+2)(A_1+a)+\Delta a\right\rbrace .
		\end{split}
	\end{equation}
	Now, we need to consider the following three cases:
	
	\noindent\textbf{Case 1.} The case $w\geq\frac{n(A_1+a)}{2A_1}>0$.
	
	Since $ w\geq\frac{n(A_1+a)}{2A_1}>0 $, there holds true
	\begin{equation*}
		\frac{2A_1^2w^2}{n}-A(A+a)w\geq0\quad\quad \mbox{and}\quad\quad G\geq G-(A_1+a)w>0,
	\end{equation*}
	then, from  (\ref{2.19}) we obtain
	\begin{equation*}
		\begin{split}
			BG\geq
			\frac{2\phi G^2}{n}-2G^{\frac{3}{2}}\left|\nabla\phi \right| +
			\phi G
			\left\lbrace -\frac{4A_1w}{n}+\dfrac{4(b-M_1)}{n}-2K-2a-2\right\rbrace ,
		\end{split}
	\end{equation*}
	i.e.\\
	\begin{equation}\label{2.20}
		\begin{split}
			B\geq
			\frac{2\phi G}{n}-2G^{\frac{1}{2}}\left|\nabla\phi \right| +
			\phi \left\lbrace -\frac{4A_1w}{n}+\dfrac{4(b-M_1)}{n}-2K-2a-2\right\rbrace .
		\end{split}
	\end{equation}
	Considering that $ w\leq\dfrac{G}{A_1+a} $, we can rewrite  (\ref{2.20}) as
	\begin{equation}\label{2.21}
		\begin{split}
			B\geq
			\frac{2\phi G}{n}-2G^{\frac{1}{2}}\left|\nabla\phi \right| -\frac{4A_1G\phi}{n(A_1+a)}+
			\phi \left\lbrace \dfrac{4(b-M_1)}{n}-2K-2a-2\right\rbrace .
		\end{split}
	\end{equation}
	Since $ a\geq2A_1>0 $, by Young's inequality, there holds
	\begin{equation}\label{2.22}
		2G^\frac{1}{2}\left|\nabla\phi \right|\leq\frac{\phi G(3a-5A_1)}{2n(A_1+a)}+\frac{\left|\nabla\phi \right|^22n(A_1+a)}{\phi(3a-5A_1)},
	\end{equation}
	then, from  (\ref{2.12}) and substituting (\ref{2.22}) into (\ref{2.21}), we have
	\begin{equation}\label{2.23}
		B\geq
		\frac{\phi G}{2n}-\frac{2n(A_1+a)C_1^2}{(3a-5A_1)R^2}+
		\phi \left\lbrace \frac{4(b-M_1)}{n}-2K-2a-2\right\rbrace .
	\end{equation}
	Also noting that $A_1+a\leq\frac{3a}{2}$, then hence $$ \frac{2n(A_1+a)}{(3a-5A_1)}\leq\frac{2n\frac{3a}{2}}{6A_1-5A_1}=\frac{3na}{A_1},$$
	so, from  (\ref{2.23}) we can derive
	\begin{equation}\label{2.24}
		B\geq
		\frac{\phi G}{2n}-\frac{3naC_1^2}{A_1R^2}+
		\phi \left\lbrace \frac{4(b-M_1)}{n}-2K-2a-2\right\rbrace .
	\end{equation}
	Noting that
	\begin{equation}\label{2.25}
		\mathop G\limits_{B_p(R)}\leq\sup\limits_{B_p(2R)}\phi G=\phi G(x_0),
	\end{equation}
	therefore, on $ B_p(R) $ we have
	\begin{equation}\label{2.26}
		G\leq2n\left\lbrace B+\frac{3n(a)^+C_1^2}{A_1R^2}+2K+2(a)^++2+\frac{4(M_1-b)^+}{n}\right\rbrace.
	\end{equation}
	
	\noindent\textbf{Case 2.} The case $0\leq w<\frac{(A_1+a)n}{2A_1}$.
	
	For this case, $\frac{2A_1^2w^2}{n}\geq0 $ and $w\leq \frac{G}{A_1+a}$ by the definition of G, hence $A_1(A_1+a)w\leq A_1G$. So we have
	\begin{equation*}
		\begin{split}
			BG\geq
			\frac{2\phi G^2}{n}-2G^{\frac{3}{2}}\left|\nabla\phi \right| +
			\phi G
			\left\lbrace -\frac{4A_1w}{n}+\dfrac{4(b-M_1)}{n}-2K-2a-2-A_1\right\rbrace ,
		\end{split}
	\end{equation*}
	then, by the same discussions as in Case 1, it holds that on $ B_p(R)$
	\begin{equation}\label{2.27}
		G\leq2n\left\lbrace B+\frac{3n(a)^+C_1^2}{A_1R^2}+2K+2(a)^++2+\frac{4(M_1-b)^+}{n}+A_1\right\rbrace.
	\end{equation}
	
	\noindent\textbf{Case 3.} The case $w<0$.
	
	Since $w<0$, it follows that $$\frac{2A_1^2w^2}{n}-A_1(A_1+a)w>0.$$
	By  (\ref{2.19}) there holds true
	\begin{equation}\label{2.28}
		\begin{split}
			BG\geq&
			\frac{2\phi G^2}{n}-2G(G-(A_1+a)w)^{\frac{1}{2}}\left|\nabla\phi \right| +\phi G\left\lbrace \dfrac{4(b-M_1)}{n}-2K-2a-2\right\rbrace\\
			& +\phi w\left\lbrace \frac{-4A_1G}{n}+\frac{4(M_1-b)A_1}{n}+(2K+2)(A_1+a)+\Delta a\right\rbrace.
		\end{split}
	\end{equation}
	It follows by Young's inequality that
	\begin{equation*}
		2G(G-(A_1+a)w)^{\frac{1}{2}}\left|\nabla\phi \right|\leq\frac{R\left|\nabla\phi \right|^2G^{\frac{3}{2}}}{C_1\phi^{\frac{1}{2}}}+\frac{C_1\phi^{\frac{1}{2}}G^{\frac{1}{2}}(G-(A_1+a)w)}{R}.
	\end{equation*}
	Then  (\ref{2.28}) can be written as
	\begin{equation}\label{2.29}
		\begin{split}
			BG\geq&
			\frac{2\phi G^2}{n}-\frac{R\left|\nabla\phi \right|^2G^{\frac{3}{2}}}{C_1\phi^{\frac{1}{2}}}-\frac{C_1\phi^{\frac{1}{2}}G^{\frac{3}{2}}}{R}+\frac{C_1\phi^{\frac{1}{2}}G^{\frac{1}{2}}(A_1+a)w}{R}\\
			& +\phi w\left\lbrace \frac{-4A_1G}{n}+\frac{4(M_1-b)A_1}{n}+(2K+2)(A_1+a)+\Delta a\right\rbrace\\
			&+\phi G\left\lbrace \frac{4(b-M_1)}{n}-2K-2a-2\right\rbrace.\\
			\\
		\end{split}
	\end{equation}
	
	Now, if
	\begin{equation*}
		\frac{C_1\phi^{\frac{1}{2}}G^{\frac{1}{2}}(A_1+a)w}{R}+
		\phi w\left\lbrace \frac{-4A_1G}{n}+\frac{4(M_1-b)A_1}{n}+(2K+2)(A_1+a)+\Delta a\right\rbrace \geq 0 ,
	\end{equation*}
	then we can drop the above terms in (\ref{2.29}) to get that
	\begin{equation*}
		BG\geq
		\frac{2\phi G^2}{n}-\frac{R\left|\nabla\phi \right|^2G^{\frac{3}{2}}}{C_1\phi^{\frac{1}{2}}}-\frac{C_1\phi^{\frac{1}{2}}G^{\frac{3}{2}}}{R}+
		\phi G\left\lbrace \dfrac{4(b-M_1)}{n}-2K-2a-2\right\rbrace.
	\end{equation*}
	i.e.\\
	\begin{equation*}
		B\geq
		\frac{2\phi G }{n}-\frac{R\left|\nabla\phi \right|^2G^{\frac{1}{2}}}{C_1\phi^{\frac{1}{2}}}-\frac{C_1\phi^{\frac{1}{2}}G^{\frac{1}{2}}}{R}+
		\phi \left\lbrace \dfrac{4(b-M_1)}{n}-2K-2a-2\right\rbrace,
	\end{equation*}
	then, it follows
	\begin{equation*}
		B\geq\frac{2\phi G }{n}-\frac{R\left|\nabla\phi \right|^2(\phi G)^{\frac{1}{2}}}{C_1\phi}-\frac{C_1\phi^{\frac{1}{2}}G^{\frac{1}{2}}}{R}+
		\phi \left\lbrace \dfrac{4(b-M_1)}{n}-2K-2a-2\right\rbrace.
	\end{equation*}
	By  (\ref{2.12}), there holds
	\begin{equation*}
		B\geq
		\frac{2\phi G }{n}-\frac{C_1(\phi G)^{\frac{1}{2}}}{R}-\frac{C_1\phi^{\frac{1}{2}}G^{\frac{1}{2}}}{R}+
		\phi \left\lbrace \dfrac{4(b-M_1)}{n}-2K-2a-2\right\rbrace,
	\end{equation*}
	consequently,
	\begin{equation*}
		B\geq
		\frac{2\phi G }{n}-\frac{2C_1(\phi G)^{\frac{1}{2}}}{R}+
		\phi \left\lbrace \dfrac{4(b-M_1)}{n}-2K-2a-2\right\rbrace.
	\end{equation*}
	Using Young's inequality again leads to
	\begin{equation*}
		\frac{2C_1(\phi G)^{\frac{1}{2}}}{R}\leq\frac{\phi G}{n}+\frac{nC_1^2}{R^2},
	\end{equation*}
	then we get
	\begin{equation*}
		B\geq
		\frac{\phi G }{n}-\frac{nC_1^2}{R^2}+
		\phi \left\lbrace \frac{4(b-M_1)}{n}-2K-2a-2\right\rbrace.
	\end{equation*}
	Finally on $ B_p(R) $, there holds true
	\begin{equation}\label{2.30}
		G\leq n\left\lbrace B+\frac{nC_1^2}{R^2}+ \frac{4(M_1-b)^+}{n}+2K+2(a)^++2\right\rbrace.\\\\\\
	\end{equation}
	
	On the other hand, if
	\begin{equation*}
		\frac{C_1\phi^{\frac{1}{2}}G^{\frac{1}{2}}(A_1+a)w}{R}+
		\phi w\left\lbrace \frac{-4A_1G}{n}+\frac{4(M_1-b)A_1}{n}+(2K+2)(A_1+a)+\Delta a\right\rbrace \leq0 ,
	\end{equation*}
	noting $w<0$ then we know there holds
	\begin{equation*}
		\frac{C_1\phi^{\frac{1}{2}}G^{\frac{1}{2}}(A_1+a)}{R}+
		\phi \left\lbrace \frac{-4A_1G}{n}+\frac{4(M_1-b)A_1}{n}+(2K+2)(A_1+a)+\Delta a\right\rbrace \geq 0.
	\end{equation*}
	By Young's inequality we have
	\begin{equation*}
		\frac{C_1\phi^{\frac{1}{2}}G^{\frac{1}{2}}(A_1+a)}{R}\leq\frac{1}{2}\left( \frac{6A_1\phi G}{n}+\frac{2C_1^2(A_1+a)^2}{3A_1R^2}\right) ,
	\end{equation*}
	hence,
	\begin{equation*}
		\frac{A_1\phi G}{n}\leq\frac{C_1^2(A_1+a)^2}{3A_1R^2}+\phi \left\lbrace \frac{4(M_1-b)A_1}{n}+(2K+2)(A_1+a)+\Delta a\right\rbrace.
	\end{equation*}
	Since we have assumed that on $B_p(2R)$
	\begin{equation*}
		\frac{4(M_1-b)A_1}{n}+(2K+2)(A_1+a)+\Delta a\geq0,
	\end{equation*}
	so on $ B_p(R) $ we have
	\begin{equation}\label{2.31}
		G\leq n\left\lbrace \frac{(A_1+(a)^+)^2C_1^2}{3A_1^2R^2}+\frac{4(M_1-b)^+}{n}+\frac{(2K+2)(A_1+a)^+}{A_1}+\frac{(\Delta a)^+}{A_1}\right\rbrace.
	\end{equation}
	By combining  (\ref{2.26}),  (\ref{2.27}),  (\ref{2.30}) and  (\ref{2.31}) and noting $ a\geq2A_1>0 $, we complete the proof of (\ref{1.3}).
	\medskip
	
	Now, we turn to the proof of (2) of Theorem \ref{thm1}.
	
	Since $A_3\leq a\leq 2A_2<0$, $\left|\nabla a \right|$, $\left|\nabla b \right|$ are bounded, $ \Delta a $ is bounded from above and $ b $ has a upper bound,  there exists a nonnegative constant $ M_2 $ such that on $ B_p(2R) $
	\begin{equation*}
		\frac{2(M_2-b)^2}{n}+(A_2+a)(M_2-b)+(2-2A_2)M_2+2KM_2-\left| \nabla a\right|^2- \left| \nabla b\right|^2\geq0,
	\end{equation*}
	and
	\begin{equation*}
		\frac{4(M_2-b)A_2}{n}+(2K+2)(A_2+a)+\Delta a\leq0.
	\end{equation*}
	Therefore, it follows from  (\ref{2.18}) that
	\begin{equation}\label{2.32}
		\begin{split}
			BG\geq&
			\frac{2\phi G^2}{n}-2G(G-(A_2+a)w)^{\frac{1}{2}}\left|\nabla\phi \right|+\phi\left( \frac{2A_2^2w^2}{n}-A_2(A_2+a)w\right)  \\
			&+\phi G\left\lbrace -\frac{4A_2w}{n}+\dfrac{4(b-M_2)}{n}-2K-2a-2\right\rbrace\\
			&+\phi w\left\lbrace \frac{4(M_2-b)A_2}{n}+(2K+2)(A_2+a)+\Delta a\right\rbrace .
		\end{split}
	\end{equation}
	
	Now, we also need to discuss the following three situations case by case:
	
	\noindent\textbf{Case 4.} The case $ w\leq0 $.
	
	Since $ w\leq0 $, it is easy to see
	\begin{equation*}
		\phi\left( \frac{2A_2^2w^2}{n}-A_2(A_2+a)w\right) \geq0,
	\end{equation*}
	and
	\begin{equation*}
		\left\lbrace \frac{4(M_2-b)A_2}{n}+(2K+2)(A_2+a)+\Delta a\right\rbrace \phi w\geq0.
	\end{equation*}
	Letting $A=A_2$, it follows from (\ref{2.32}) that
	\begin{equation*}
		BG\geq
		\frac{2\phi G^2}{n}-2G(G-(A_2+a)w)^{\frac{1}{2}}\left|\nabla\phi \right|+\phi G\left\lbrace -\frac{4A_2w}{n}+\dfrac{4(b-M_2)}{n}-2K-2\right\rbrace.
	\end{equation*}
	Noting also that
	\begin{equation*}
		0<A_2w\leq\frac{A_2G}{A_2+a},
	\end{equation*}
	then we have
	\begin{equation*}
		BG\geq
		\frac{2\phi G^2}{n}-2G(G-(A_2+a)w)^{\frac{1}{2}}\left|\nabla\phi \right| +\phi G\left\lbrace -\frac{4A_2G}{n(A_2+a)}+\frac{4(b-M_2)}{n}-2K-2\right\rbrace,
	\end{equation*}
	i.e.
	\begin{equation*}
		\begin{split}
			B\geq&
			\frac{2\phi G}{n}-2G^{\frac{1}{2}}\left|\nabla\phi \right| +\phi \left\lbrace -\frac{4A_2G}{n(A_2+a)}+\dfrac{4(b-M_2)}{n}-2K-2\right\rbrace.
		\end{split}
	\end{equation*}
	Since $ A_2+a<0 $ and $ 3a-5A_2<0 $, it follows from Young's inequality that
	\begin{equation*}
		2G^{\frac{1}{2}}\left|\nabla\phi \right|\leq\frac{\phi G(3a-5A_2)}{2n(A_2+a)}+\frac{\left|\nabla\phi \right|^22n(A_2+a)}{\phi(3a-5A_2)}.
	\end{equation*}
	Hence, in view of  (\ref{2.12}) we can infer from the above two inequalities that
	\begin{equation}\label{2.33}
		B\geq
		\frac{\phi G}{2n}-\frac{2n(A_2+a)C_1^2}{(3a-5A_2)R^2} +
		\phi \left\lbrace \frac{4(b-M_2)}{n}-2K-2\right\rbrace.
	\end{equation}
	Obviously $2A_2\leq a<0 $ also implies that $$\frac{2n(A_2+a)}{3a-5A_2}\leq\frac{3na}{A_2},$$
	therefore, on $B_p(R)$ we have
	\begin{equation}\label{2.34}
		G\leq2n\left\lbrace B+\frac{3n(a)^+C_1^2}{A_2R^2}+2K+2+\frac{4(M_2-b)^+}{n} \right\rbrace.
	\end{equation}
	
	\noindent\textbf{Case 5.} The case $ 0\leq w\leq \frac{n(A_2+a)}{2A_2}$.
	
	In the present situation, from  (\ref{2.32}) we can derive
	\begin{equation}\label{2.35}
		\begin{split}
			BG\geq&
			\frac{2\phi G^2}{n}-2G(G-(A_2+a)w)^{\frac{1}{2}}\left|\nabla\phi \right|\\
			&+\phi G\left\lbrace -\frac{4A_2w}{n}+\dfrac{4(b-M_2)}{n}-2K-2a-2\right\rbrace\\
			&+\phi w\left\lbrace \frac{4(M_2-b)A_2}{n}+(2K+2)(A_2+a)+\Delta a-A_2(A_2+a)\right\rbrace .
		\end{split}
	\end{equation}
	Note that  $ (A_2+a)w\leq0 $, and by Young's inequality we have
	\begin{equation*}
		2G(G-(A_2+a)w)^{\frac{1}{2}}\left|\nabla\phi \right|\leq\frac{R\left|\nabla\phi \right|^2G^{\frac{3}{2}}}{C_1\phi^{\frac{1}{2}}}+\frac{C_1\phi^{\frac{1}{2}}G^{\frac{1}{2}}(G-(A_2+a)w)}{R}.
	\end{equation*}
	Therefore, we can rewrite  (\ref{2.35}) as
	\begin{equation}\label{2.36}
		\begin{split}
			BG\geq&
			\frac{2\phi G^2}{n}-\frac{R\left|\nabla\phi \right|^2G^{\frac{3}{2}}}{C_1\phi^{\frac{1}{2}}}-\frac{C_1\phi^{\frac{1}{2}}G^{\frac{3}{2}}}{R}+
			\phi G\left\lbrace\dfrac{4(b-M_2)}{n}-2K-2a-2\right\rbrace\\
			&+\phi w\left\lbrace\frac{4(M_2-b)A_2}{n} +(2K+2)(A_2+a)+\Delta a-A_2(A_2+a)\right\rbrace \\
			&+\phi w\left\lbrace \frac{C_1\phi^{\frac{1}{2}}G^{\frac{1}{2}}(A_2+a)}{R}-\frac{4A_2G}{n}\right\rbrace .
		\end{split}
	\end{equation}
	Hence, if
	\begin{equation*}
		\begin{split}
			&\phi w\left\lbrace\frac{4(M_2-b)A_2}{n} +(2K+2)(A_2+a)+\Delta a-A_2(A_2+a)\right\rbrace \\
			&+\phi w\left\lbrace \frac{C_1\phi^{\frac{1}{2}}G^{\frac{1}{2}}(A_2+a)}{R}-\frac{4A_2G}{n}\right\rbrace \geq0,
		\end{split}
	\end{equation*}
	we know that there holds true
	\begin{equation*}
		\begin{split}
			&	BG\geq
			\frac{2\phi G^2}{n}-\frac{R\left|\nabla\phi \right|^2G^{\frac{3}{2}}}{C_1\phi^{\frac{1}{2}}}-\frac{C_1\phi^{\frac{1}{2}}G^{\frac{3}{2}}}{R}+
			\phi G\left\lbrace\frac{4(b-M_2)}{n}-2K-2a-2\right\rbrace.
		\end{split}
	\end{equation*}
	By (\ref{2.12}) again and the above expression we can see that there holds
	\begin{equation*}
		\begin{split}
			&	BG\geq
			\frac{2\phi G^2}{n}-\frac{2C_1\phi^{\frac{1}{2}}G^{\frac{3}{2}}}{R}+
			\phi G\left\lbrace\dfrac{4(b-M_2)}{n}-2K-2a-2\right\rbrace,
		\end{split}
	\end{equation*}
	i.e.
	\begin{equation}\label{2.37}
		\begin{split}
			&	B\geq
			\frac{2\phi G}{n}-\frac{2C_1\phi^{\frac{1}{2}}G^{\frac{1}{2}}}{R}+
			\phi \left\lbrace\dfrac{4(b-M_2)}{n}-2K-2a-2\right\rbrace.
		\end{split}
	\end{equation}
	Now, by the same arguments as deducing  (\ref{2.30}) from  (\ref{2.29}) we can easily see that there holds that on $B_p(R)$
	\begin{equation}\label{2.38}
		G\leq n\left\lbrace B+\frac{nC_1^2}{R^2}+ \frac{4(M_2-b)^+}{n}+2K+2(a)^++2\right\rbrace.\\\\\\
	\end{equation}
	
	On the other hand, if
	\begin{equation*}
		\begin{split}
			&\phi w\left\lbrace\frac{4(M_2-b)A_2}{n} +(2K+2)(A_2+a)+\Delta a-A_2(A_2+a)\right\rbrace\\
			&+\phi w\left\lbrace \frac{C_1\phi^{\frac{1}{2}}G^{\frac{1}{2}}(A_2+a)}{R}-\frac{4A_2G}{n}\right\rbrace \leq0,
		\end{split}
	\end{equation*}
	then, noting $w\geq 0$ we have
	\begin{equation*}
		\begin{split}
			&\frac{4(M_2-b)A_2}{n} +(2K+2)(A_2+a)+\Delta a-A_2(A_2+a)\\
			&+\frac{C_1\phi^{\frac{1}{2}}G^{\frac{1}{2}}(A_2+a)}{R}-\frac{4A_2G}{n} \leq0.
		\end{split}
	\end{equation*}
	Since $ a<0 $ and $ A_2+a<0$, by multiplying $-1$ and adopting almost the same discussions as those to obtain (\ref{2.31}) in case 3, on $B_p(R)$ we have
	\begin{equation}\label{2.39}
		G\leq n\left\lbrace \frac{(A_2+(a)^+)^2C_1^2}{3A_2^2R^2}+\frac{4(M_2-b)^+}{n}+\frac{(2K+2)(A_2+a)^+}{A_2}+\frac{(\Delta a)^+}{A_2}-(A_2+a)^+\right\rbrace.
	\end{equation}
	
	\noindent\textbf{Case 6.} The case $w\geq\frac{(A_2+a)n}{2A_2}$.
	
	For this situation, $ A_2(A_2+a)w<0 $ and $ \frac{2A_2^2w^2}{n}-A_2(A_2+a)w\geq0 $, then by  (\ref{2.32}), there holds
	\begin{equation}\label{2.40}
		\begin{split}
			BG\geq&
			\frac{2\phi G^2}{n}-2G(G-(A_2+a)w)^{\frac{1}{2}}\left|\nabla\phi \right|\\
			&+\phi G\left\lbrace -\frac{4A_2w}{n}+\dfrac{4(b-M_2)}{n}-2K-2a-2\right\rbrace\\
			&+\phi w\left\lbrace \frac{4(M_2-b)A_2}{n}+(2K+2)(A_2+a)+\Delta a\right\rbrace .
		\end{split}
	\end{equation}
	Noting that after removing the term $A_2(A_2+a)w$ in  (\ref{2.35}), then it becomes the same as  (\ref{2.40}), so all the discussions here will be almost the same as those in step 5. So the results here are almost the same as those in Case 5. As a consequece, we have
	\begin{equation}\label{2.41}
		G\leq n\left\lbrace B+\frac{nC_1^2}{R^2}+ \frac{4(M_2-b)^+}{n}+2K+2(a)^++2\right\rbrace
	\end{equation}
which is similar to (\ref{2.38}) or
	\begin{equation}\label{2.42}
		G\leq n\left\lbrace \frac{(A_2+(a)^+)^2C_1^2}{3A_2^2R^2}+\frac{4(M_2-b)^+}{n}+\frac{(2K+2)(A_2+a)^+}{A_2}+\frac{(\Delta a)^+}{A_2}\right\rbrace
	\end{equation}
which is similar to (\ref{2.39}).
	
Eventually, since $\frac{a}{A_2}\geq2$, then by combining  (\ref{2.34}),  (\ref{2.38}),  (\ref{2.39}) and the arguments in Case 6, we know that  (\ref{1.4}) holds true. Thus we accomplish the proof of Theorem 1.1.
\end{proof}

Next, we give proof of Corollary {\ref{c1.1}}.
\begin{proof}
	(1) Letting R$\rightarrow  +\infty$, then from  (\ref{1.3}), there holds
	\begin{equation*}
		(A_1+a)f
		\leq n\left\lbrace \left( 3K+3\right) \left(\frac{A_4}{A_1} \right)+\frac{8}{n}\left( M_1-b_1\right)+5A_4\right\rbrace
	\end{equation*}
	since $\left| \nabla f\right|^2\geq0$ and $M_1\geq0$.
	Then just using $A_1+a\geq 3A_1$ we can see that the required estimate follows at once. The proof of (2) is just the same as (1).
\end{proof}

\section{Proof of Theorem 1.2}
The main methods here are similar to those in the proof of Theorem 1.1. First of all, we need to establish the following lemma.

\begin{lemma}
	Let $u\leq D$ be a positive smooth solution of (\ref{equ*}) for some positive constant $D$. Let $f=\log\frac{u}{D}$, $F=t\left\lbrace \left|\nabla f\right|^2 +(A+a)f+2(M+b)-2f_t\right\rbrace$ for some constants $A$, $M$ to be determined later, and $h= \frac{\left| f\right|^2}{F}$. Then, on $B_p(2R)\times(0,\infty)$ there holds
\begin{equation}\label{3.1}
\begin{split}
\Delta F-F_t \geq & t\left\lbrace \left( A-2K-2-\left| \log D\right|\right) hF+\frac{(1+ht)^2F^2}{2nt^2}-\frac{2(1+ht)\left( M-a\log D\right) F}{nt}\right\rbrace\\
&+tf\left\lbrace\frac{(1+ht)(a-A)F}{nt}+\Delta a+a_t+\frac{2\left( A-a\right) \left( M-a\log D\right) }{n} \right\rbrace\\
&+t\left\lbrace 2a(M+b)+2a_t\log D+2\Delta b+\frac{2(M-a\log D)^2}{n}\right\rbrace-\frac{F}{t}-aF \\
&-t\left\lbrace \left( A+a\right) \left( a\log D+b\right) +\left|  \nabla b\right|^2+\left( 1+\left| \log D\right|\right) \left|\nabla a \right|^2 \right\rbrace-2\left\langle \nabla f,\nabla F\right\rangle.
\end{split}
\end{equation}
\end{lemma}

\begin{proof}
	By direct calculations, we derive the following identities
	\begin{equation}\label{3.2}
		\left| \nabla f\right|^2=\frac{F}{t}-(A+a)f-2(M+b)-2f_t,
	\end{equation}
	\begin{equation}\label{3.3}
		\Delta f=f_t-af-a\log D-b-\left| \nabla f\right|^2,
	\end{equation}
	and\\
	\begin{equation}\label{3.4}
		f_{tt}=\Delta f_t+af_t+2\left\langle\nabla f,\nabla f_t \right\rangle+b_t+a_tf+a_t\log D.
	\end{equation}
	By the definition of $F$, (\ref{3.3}) is equivalent to
	\begin{equation}\label{3.5}
		\Delta f=\frac{-F}{2t}-\frac{(A-a)f}{2}-\frac{\left| \nabla f\right|^2}{2}+M-a\log D.
	\end{equation}
	By virtue of Bochner Formula, there holds
	\begin{equation}\label{3.6}
		\begin{split}
			\Delta F&=t\left\lbrace \Delta\left| \nabla f\right|^2+(A+a)\Delta f+f\Delta a+2\left\langle \nabla a,\nabla f\right\rangle+2\Delta b-2\Delta f_t \right\rbrace\\
			&=t\left\lbrace2\left\langle \nabla f,\nabla\Delta f\right\rangle+2\left| D^2f\right|^2+2Ric\left\langle \nabla f,\nabla f \right\rangle\right\rbrace \\
			&\hspace*{1em}+t\left\lbrace f\Delta a+(A+a)\Delta f+2\left\langle \nabla a,\nabla f\right\rangle+2\Delta b-2\Delta f_t    \right\rbrace .
		\end{split}
	\end{equation}
	Next, substituting (\ref{3.2}) into (\ref{3.5}), we have
	\begin{equation*}
		\Delta f=\frac{-F}{t}+Af-f_t+2M+b-a\log D,
	\end{equation*}
	then we get
	\begin{equation*}
		\left\langle \nabla f,\nabla\Delta f\right\rangle=-\left\langle\nabla f,\frac{F}{t}\right\rangle+A\left| \nabla f\right|^2-\left\langle \nabla f,\nabla f_t\right\rangle+\left\langle \nabla f, \nabla b\right\rangle-\left\langle \nabla f,\nabla a \right\rangle\log D.
	\end{equation*}
	By Cauchy-Schwarz inequality, on $B_p(2R)$ it holds that
	\begin{equation*}
		\left| D^2f\right|^2\geq\frac{(\Delta f)^2}{n},
	\end{equation*}
	\begin{equation*}
		2\left\langle \nabla f,\nabla b\right\rangle\geq -(\left| \nabla f\right|^2+\left| \nabla b\right|^2),
	\end{equation*}
	and
	\begin{equation*}
		\left( 2-2\log D\right) \left\langle \nabla f,\nabla a\right\rangle\geq\left( -1-\left| \log D\right|\right) \left( \left| \nabla f\right|^2+\left| \nabla a\right|^2\right) .
	\end{equation*}
	Then, from (\ref{3.6}) we derive
	\begin{equation}\label{3.7}
		\begin{split}
			\Delta F\geq&-2\left\langle \nabla f,\nabla F\right\rangle +t\left\lbrace\left( 2A-2K-2-\left| \log D\right|\right) \left| \nabla f\right|^2+\frac{2(\Delta f)^2}{n}\right\rbrace\\
			&+t\left\lbrace f\Delta a+(A+a)\Delta f+2\Delta b-2\Delta f_t-2\left\langle \nabla f,\nabla f_t\right\rangle \right\rbrace \\
			&-t\left\lbrace \left| \nabla b\right|^2+\left( 1+\left|\log D \right|\right) \left| \nabla a\right|^2\right\rbrace. \\
		\end{split}
	\end{equation}
	
	As for $F_t$, first of all, we have:
	\begin{equation}\label{3.8}
		F_t=\frac{F}{t}+t\left\lbrace 2\left\langle \nabla f,\nabla f_t\right\rangle+(A+a)f_t+a_tf+2b_t-2f_{tt} \right\rbrace.
	\end{equation}
	By combining (\ref{3.3}), (\ref{3.4}), (\ref{3.7}), (\ref{3.8}) and also noting that
	\begin{equation*}
		-aF=t\left\lbrace -a\left| \nabla f\right|^2-(A+a)af-2a(M+b)+2af_t\right\rbrace,
	\end{equation*}
	we can see that there holds true
	\begin{equation}\label{3.9}
		\begin{split}
			&\Delta F-F_t\\
			&\geq-2\left\langle \nabla f,\nabla F\right\rangle-\frac{F}{t}-aF+t\left\lbrace \left( A-2K-2-\left| \log D\right|\right) \left| \nabla f\right|^2 \right\rbrace\\
			&\hspace*{1.2em}+t\left\lbrace \frac{2(\Delta f)^2}{n}+f\Delta a+a_tf+2a_t\log D+2a(M+b)+2\Delta b\right\rbrace\\
			&\hspace*{1.2em}-t\left\lbrace (A+a)\left( a\log D+b\right) +\left| \nabla b\right|^2+\left( 1+\left|\log D \right|\right) \left| \nabla a\right|^2\right\rbrace.
		\end{split}
	\end{equation}
	Next, substituting $hF=\left| \nabla f\right|^2$ into (\ref{3.5}) and noting that $\left( \frac{(A-a)f}{2}\right)^2\geq0$, then we deduce that
	\begin{equation}\label{3.10}
		\begin{split}
			(\Delta f)^2=&\left(  \frac{-(1+ht)F}{2t}+\frac{(A-a)f}{2}+\left( M-a\log D\right) \right)^2\\
			\geq&\frac{(1+ht)^2F^2}{4t^2}+(M-a\log D)^2-\frac{(1+ht)(A-a)Ff}{2t}\\
			&-\frac{(1+ht)(M-a\log D)F}{t}+(A-a)(M-a\log D)f.
		\end{split}
	\end{equation}
Then, just substituting (\ref{3.10}) into (\ref{3.9}) we obtain (\ref{3.1}) at once.
\end{proof}

Now we are going to give the proof of Theorem 1.2, and it is divided into two parts. Without loss of generality, we assume $F>0$, since otherwise the result will be trivial.
\begin{proof}[\textbf{Proof of Theorem 1.2} ]
	Since on $B_{p}(2R)\times (0,T]$ for any $T>0$, $a , b ,a_t, |\nabla a|$ and $|\nabla b|$ are bounded and $\Delta b$ has a lower bound, then there exists a constant $M$ such that
	\begin{equation}\label{3.11}
		\left\{
		\begin{array}{lr}	
			M+b\geq0,&\\
			M-a\log D\geq0,&\\		
			\frac{2(M-a\log D)^2}{n}+2a(M+b)-(A+a)(a\log D+b)&\\
			+2a_t\log D+2\Delta b-\left\lbrace \left|  \nabla b\right|^2+\left( 1+\left| \log D\right|\right) \left|\nabla a \right|^2 \right\rbrace\geq0.& 	
		\end{array}
		\right.
	\end{equation}
Consequently, letting $A$ be a constant strictly larger than $[a]^+$, then from (\ref{3.1}) we obtain
\begin{equation}\label{3.12}
\begin{split}
\Delta F-F_t\geq &t\left\lbrace (A-2K-2-\left| \log D\right| )hF+\frac{(1+ht)^2F^2}{2nt^2}-\frac{2(1+ht)(M-a\log D)F}{nt}\right\rbrace\\
&+tf\left\lbrace\frac{(1+ht)(a-A)F}{nt}+\Delta a+a_t+\frac{2(A-a)(M-a\log D)}{n} \right\rbrace \\
&-2\left\langle \nabla f,\nabla F\right\rangle-\frac{F}{t}-aF.
\end{split}
\end{equation}
	
Next, we consider the same cut-off function $ \phi $ as in Section 2, and denote $\phi F$ by $\lambda$. Let $(x_0, t_0)\in B_p(2R)\times (0,T]$ be the maximum point of $\lambda$, then
\begin{equation*}
\nabla\lambda(x_0, t_0)=0, \hspace*{1em} \Delta\lambda(x_0, t_0)\leq0,\hspace*{1em}  \mbox{and} \hspace*{1em} F_t\geq0.
\end{equation*}
Then, by the same discussions as to obtain (\ref{2.16}), we also have $BF\geq\phi\Delta F$ and $\phi\nabla F=-F\nabla\phi$ at $(x_0, t_0)$. From now on, all the discussions are considered at the point $(x_0, t_0) $, for simplicity, we still write $t_0$ as t and omit $x_0$.
	
By the above discussions, (\ref{3.12}) changes into:
	\begin{equation}\label{3.13}
		\begin{split}
			BF&\geq\phi\Delta F\\
			&\geq\phi t\left\lbrace (A-2K-2-\left| \log D\right| )hF+\frac{(1+ht)^2F^2}{2nt^2}-\frac{2(1+ht)(M-a\log D)F}{nt}\right\rbrace\\
			&\hspace*{1.2em}+\phi tf \left\lbrace\frac{(1+ht)(a-A)F}{nt}+\Delta a+a_t+\frac{2(A-a)(M-a\log D)}{n} \right\rbrace \\
			&\hspace*{1.2em}-2\phi\left\langle \nabla f,\nabla F\right\rangle-\frac{\phi F}{t}-a\phi F.
		\end{split}
	\end{equation}
	
Now, at $(x_0,t_0)$ we need to consider the following cases:
	
	\noindent\textbf{Case I.} The case
	\begin{equation*}
		f\left\lbrace\frac{(1+ht)(a-A)F}{nt}+\Delta a+a_t+\frac{2(A-a)(M-a\log D)}{n} \right\rbrace \leq 0.
	\end{equation*}
	
	For the present situation, since $f\leq0$, we have
	\begin{equation*}
		\frac{(1+ht)(a-A)F}{nt}+\Delta a+a_t+\frac{2(A-a)(M-a\log D)}{n}\geq0.
	\end{equation*}	
	Noting $A>[a]^+$, we derive
	\begin{align}
		F&\leq\frac{nt}{(A-a)(1+ht)}\left( \Delta a+a_t+\frac{2(A-a)(M-a\log D)}{n}\right)\nonumber\\
		&\leq \frac{nt}{(A-a)}\left( \Delta a+a_t+\frac{2(A-a)(M-a\log D)}{n}\right)\nonumber\\
		&\leq\frac{nT[\Delta a+a_t]^+}{(A-[a]^+)} +2T[M-a\log D]^+.\label{3.14}
	\end{align}	
In fact, if furthermore assuming that $$\Delta a+a_t+\frac{2}{n}(A-a)(M-a\log D)\leq 0,$$
then we have that there holds true
\begin{equation*}
f\left\lbrace\frac{(1+ht)(a-A)F}{nt}+\Delta a+a_t+\frac{2(A-a)(M-a\log D)}{n} \right\rbrace \geq 0,
\end{equation*}	
and this becomes a special case of Case II. So, here we always assume $$ \Delta a+a_t+\frac{2}{n}(A-a)(M-a\log D)>0$$ to ensure (\ref{3.14}) make sense.
\medskip
	
	\noindent\textbf{Case II.} The case	
	\begin{equation*}
		f\left\lbrace\frac{(1+ht)(a-A)F}{nt}+\Delta a+a_t+\frac{2(A-a)(M-a\log D)}{n} \right\rbrace \geq 0.
	\end{equation*}	
	
	Now, from (\ref{3.13}) we obtain
	\begin{equation}\label{3.15}
		\begin{split}
			BF&\geq\phi\Delta F\\
			&\geq\phi t\left\lbrace \left( A-2K-2-\left| \log D\right| \right) hF+\frac{(1+ht)^2F^2}{2nt^2}-\frac{2(1+ht)(M-a\log D)F}{nt}\right\rbrace\\
			&\hspace*{1.2em}-2\phi\left\langle \nabla f,\nabla F\right\rangle-\frac{\phi F}{t}-a\phi F.
		\end{split}
	\end{equation}
	By $ \phi\nabla F=-F\nabla\phi $, we deduce that
	\begin{equation*}
		-2\phi\left\langle \nabla f,\nabla F\right\rangle=2\left\langle\nabla f,\nabla \phi \right\rangle F,
	\end{equation*}	
	then by (\ref{2.12}) and $hF=\left| \nabla f\right|^2$, it turns out that
	\begin{align}
		2\left\langle\nabla f,\nabla \phi \right\rangle F&\geq-2\left| \nabla \phi\right| \left| \nabla f\right|F\nonumber\\
		&=-2\left( \frac{\left| \nabla \phi\right|^2}{\phi}\right)^{\frac{1}{2}}\phi^{\frac{1}{2}}F\left| \nabla f\right|\nonumber\\
		&\geq-\frac{2C_1\phi^{\frac{1}{2}}F^{\frac{3}{2}}h^\frac{1}{2}}{R}.\label{3.16}
	\end{align}	
	Moreover, we have
	\begin{equation}\label{3.17}
		\left\{
		\begin{array}{lr}
			-a\phi F\geq-\left[ a\right] ^+\phi F,&\\
			\phi t\left( A-2K-2-\left| \log D\right| \right)hF\geq \phi t\left\lbrace -\left[ -\left( A-2K-2-\left| \log D\right| \right)\right]^+ \right\rbrace hF.&
		\end{array}
		\right.
	\end{equation}
	Combining (\ref{3.16}) and (\ref{3.17}), in view of (\ref{3.15}) we deduce that
	\begin{equation}\label{3.18}
		\begin{split}
			BF\geq&
			-\frac{2C_1\phi^{\frac{1}{2}}F^{\frac{3}{2}}h^\frac{1}{2}}{R}-\frac{\phi F}{t}-\left[ a\right] ^+\phi F\\
			&+\phi t\left\lbrace -\left[ -\left( A-2K-2-\left| \log D\right| \right)\right]^+hF\right\rbrace \\
			&+\phi t\left\lbrace \frac{(1+ht)^2F^2}{2nt^2}-\frac{2(1+ht)(M-a\log D)F}{nt}\right\rbrace.
		\end{split}
	\end{equation}
	Then, multiplying by $\phi t$ on both sides of (\ref{3.18}), and noting that on $B_{2R}(p)$: $0<\phi\leq1$, $\phi^2\leq\phi$ and $M>a\log D$, from (\ref{3.18}) we deduce that
	\begin{equation*}
		\begin{split}
			Bt\lambda\geq&
			-\frac{2C_1t\lambda^{\frac{3}{2}}h^\frac{1}{2}}{R}-\lambda-\left[ a\right] ^+t\lambda-\lambda t^2h\left[ -\left( A-2K-2-\left| \log D\right| \right)\right]^+\\
			& +\frac{(1+ht)^2\lambda^2}{2n}-\frac{2t(1+ht)(M-a\log D)\lambda}{n},
		\end{split}
	\end{equation*}		
	i.e.
	\begin{equation*}
		\begin{split}
			Bt\geq&
			-\frac{2C_1t\lambda^{\frac{1}{2}}h^\frac{1}{2}}{R}-1-\left[ a\right] ^+t-t^2h\left[ -\left( A-2K-2-\left| \log D\right| \right)\right]^+\\
			&  +\frac{(1+ht)^2\lambda}{2n}-\frac{2t(1+ht)(M-a\log D)}{n},
		\end{split}
	\end{equation*}		
	hence, there holds true
	\begin{equation*}
		\begin{split}
			\frac{(1+ht)^2\lambda}{2n}\leq& Bt+\frac{2C_1t\lambda^{\frac{1}{2}}h^\frac{1}{2}}{R}
			+1+\left[ a\right] ^+t+t^2h\left[ -\left( A-2K-2-\left| \log D\right| \right)\right]^+\\
			&  +\frac{2t(1+ht)(M-a\log D)}{n}.
		\end{split}
	\end{equation*}		
	By virtue of Young's inequality we have
	\begin{equation*}
		\frac{2C_1t\lambda^{\frac{1}{2}}h^\frac{1}{2}}{R}\leq\frac{\lambda(1+ht)^2}{4n}+\frac{4nC_1^2ht^2}{R^2(1+ht)^2},
	\end{equation*}	
	consequently we derive that
	\begin{equation*}
		\begin{split}
			\frac{(1+ht)^2\lambda}{4n}\leq& Bt+1+\left[ a\right] ^+t+\frac{4nC_1^2ht^2}{R^2(1+ht)^2}+t^2h\left[ -\left( A-2K-2-\left| \log D\right| \right)\right]^+\\
			&  +\frac{2t(1+ht)(M-a\log D)}{n}.
		\end{split}
	\end{equation*}	
	It turns out that there holds:
	\begin{equation}\label{3.19}
		\begin{split}
			\lambda\leq&\frac{4n}{(1+ht)^2} \left\lbrace Bt+1+\left[ a\right] ^+t+\frac{4nC_1^2ht^2}{R^2(1+ht)^2}+t^2h\left[ -\left( A-2K-2-\left| \log D\right| \right)\right]^+\right\rbrace \\
			& +\frac{4n}{(1+ht)^2}\left\lbrace  \frac{2t(1+ht)(M-a\log D)}{n}\right\rbrace .
		\end{split}
	\end{equation}	
	Moreover, it is well-known that the following claim holds: If $x\geq0$, then, for integer $n\geq1$ there holds true that $$(1+x)^n\geq1+nx\geq nx.$$
	Then, by virtue of the claim we have
	\begin{align}
		&\frac{ht^2}{(1+ht)^4}\leq\frac{t^2h}{4ht}=\frac{t}{4},\label{3.20}\\
		&\frac{ht^2}{(1+ht)^2}\leq\frac{t^2h}{2ht}=\frac{t}{2},\label{3.21}\\
		&\frac{1+ht}{(1+ht)^2}\leq\frac{1}{1+ht}\leq1.\label{3.22}
	\end{align}	
	Combining (\ref{3.19}) to (\ref{3.22}), on $B_p(2R)\times(0, T]$ we derive
	\begin{equation}\label{3.23}
		\begin{split}
			\lambda\leq&4n \left\lbrace Bt
			+1+\left[ a\right] ^+t+\frac{4nC_1^2t}{R^2}+\frac{t}{2}\left[ -\left( A-2K-2-\left| \log D\right| \right)\right]^+\right\rbrace \\
			& +4n\left\lbrace  \frac{2t\left[ M-a\log D\right] ^+}{n}\right\rbrace\\
			\leq&4n \left\lbrace BT
			+1+\left[ a\right] ^+T+\frac{4nC_1^2T}{R^2}+\frac{T}{2}\left[ -\left( A-2K-2-\left| \log D\right| \right)\right]^+ \right\rbrace \\
			& +4n\left\lbrace \frac{2T\left[ M-a\log D\right] ^+}{n}\right\rbrace.
		\end{split}
	\end{equation}	
	Ultimately, it's obvious that the following expressions hold true
	\begin{equation}\label{3.24}
		F\mid_{B_p(r)}=\phi F\mid_{B_p(r)}=\lambda\mid_{B_p(r)}\leq\lambda(x_0,t_0).
	\end{equation}	
	Therefore combining (\ref{3.14}) and (\ref{3.23}) and also noting that $T>0$ is arbitrary, we obtain (\ref{1.5}), then we accomplish the proof of Theorem 1.2.
\end{proof}	

\begin{remark}\label{r3.1}
	Now we can give some explanations about the previous Remark \ref{r1.1}. If a is a constant, e.g. $a\equiv A$, then from (\ref{3.13}) we know that the arguments on Case 1 are completely unnecessary, so we don't have to assume $u$ has a upper bound by letting $D\equiv1$. Especially, when $a\equiv0$, from (\ref{3.11}) we know that it's even unnecessary to assume $b$ is bounded, but in this situation, $b$ will appear on the left side of (\ref{1.5}) since $M+b\geq0$ may not be true.
\end{remark}

\section{Discussions about Theorem \ref{thm3}}	
In this section, we apply our results in Theorem \ref{thm1} and \ref{thm2} to logarithmic Schr\"odinger equation (\ref{1.8}). To this end, we begin to prove Theorem \ref{thm3}.
\begin{proof}[\textbf{Proof of Theorem \ref{thm3}}]
(1)
From the beginning of proof of (1) of Theorem \ref{thm1}, we know that all our requirements about $M_1$ are
\begin{equation*}
(A_1+a)(M_1-b)+(2-2A_1)M_1+\dfrac{2(b-M_1)^2}{n}+2KM_1-\left| \nabla a\right|^2- \left| \nabla b\right|^2\geq0
\end{equation*}
and
\begin{equation*}
\frac{4(M_1-b)A_1}{n}+(2K+2)(A_1+a)+\Delta a\geq0.
\end{equation*}
For (\ref{1.9}), $a\equiv2$, $A_1=1$, $K=0$ and $b(x)=V(x)$, then the above requirements become
\begin{equation}\label{4.1}
\left\{
\begin{array}{lr}	
3(M_1-V)+\frac{2(M_1-V)^2}{n}-\left| \nabla V\right|^2\geq0,&\\	
\frac{4(M_1-b)A_1}{n}+6\geq0.& 	
\end{array}
\right.
\end{equation}
Specifically, we can take
\begin{equation}\label{4.2}
\left\{
\begin{array}{lr}	
M_1-V\geq\frac{1}{3}\sup\limits_{R^n}\left| \nabla V\right|^2,&\\	
M_1-V\geq-\frac{3n}{2}.& 	
\end{array}
\right.
\end{equation}
Then we may choose $M_1=\sup\limits_{R^n}\left| V\right| +\frac{1}{3}\sup\limits_{R^n}\left| \nabla V\right|^2$. By (1) of Theorem \ref{thm1}, letting $R\longrightarrow+\infty$, we obtain
\begin{equation*}
3f\leq16n+8(M_1-V)^+=16n+8\sup\limits_{R^n}\left| V\right|+\frac{8}{3}\sup\limits_{R^n}\left| \nabla V\right|^2-8V,
\end{equation*}
thus
\begin{equation}\label{4.3}
u\leq e^{\frac{1}{3}\left\lbrace 16n+8\sup\limits_{R^n}\left| V\right|+\frac{8}{3}\sup\limits_{R^n}\left| \nabla V\right|^2-8V\right\rbrace},
\end{equation}
i.e., if $V(x)$, $\left| \nabla V(x)\right|$ is bounded, then the positive solutions to equation (\ref{1.8}) must be bounded. Especially, when $V\geq0$ is a constant, from (\ref{4.3}) we know that the upper bound of $u$ is not related to $V$.

(2)
On the other hand, we can also apply Theorem \ref{thm2} to this equation. In this situation, we may assume $u_t=0$, then we can let $T\longrightarrow\infty$ and $R\longrightarrow+\infty$ in (\ref{1.5}). By Remark \ref{r3.1}, if we take $A=a=2, D=1, b(x)=V(x), K=0$, then $$M=\frac{1}{2}\sup\limits_{R^n}\left| \Delta V\right|+\frac{1}{4}\sup\limits_{R^n}\left| \nabla V\right|+\sup\limits_{R^n}\left| V\right| $$ ensures
(\ref{3.11}) to be nonnegative and $M+b\geq0, M\geq0$. By (\ref{1.5}), we have
\begin{equation*}
4f\leq4n\left\lbrace 2+\frac{2}{n}\left( \frac{1}{2}\sup\limits_{R^n}\left| \Delta V\right|+\frac{1}{4}\sup\limits_{R^n}\left| \nabla V\right|+\sup\limits_{R^n}\left| V\right|\right) \right\rbrace,
\end{equation*}
i.e.
\begin{equation}\label{4.4}
u\leq e^{\left\lbrace 2n+\sup\limits_{R^n}\left| \Delta V\right| +\frac{1}{2}\sup\limits_{R^n}\left| \nabla V\right|+2\sup\limits_{R^n}\left| V\right|\right\rbrace}.
\end{equation}
\end{proof}

\noindent\textbf{Acknowledgements} The author is supported partially by NSFC grant (No.11731001) and NSFC grant (No.11971400) and is grateful to Prof. Youde Wang for many constructive opinions.

\end{document}